\documentclass[review]{elsarticle}
\usepackage{amsfonts}
\usepackage{amsmath}
\usepackage{amsthm}
\usepackage{color}
\usepackage{tikz}
\usepackage{lineno,hyperref}
\modulolinenumbers[5]
\newcommand{\R}{\mathbb{R}}
\newcommand{\lap}{\mbox{$\Delta$}}
\newcommand{\grad}{\mbox{$\nabla$}}

\newcommand{\la}{\lambda}
\newcommand{\be}{\begin{equation}}
\newcommand{\ee}{\end{equation}}
\journal{Journal of \LaTeX\ Templates}
\newtheorem{thm}{Theorem}[section]
 \newtheorem{lem}{Lemma}[section]

 \newtheorem{rem}{Remark}[section]
 
\begin{document}
\begin{frontmatter}

\title{\bf Hopf's lemmas for parabolic fractional Laplacians and parabolic fractional $p$-Laplacians  }


\author[mymainaddress,mysecondaryaddress]{Pengyan Wang}
\ead{wangpy119@126.com}

\author[mymainaddress]{Wenxiong Chen\corref{mycorrespondingauthor}}
\cortext[mycorrespondingauthor]{Corresponding author}
\ead{wchen@yu.edu}

\address[mymainaddress]{Department of   Mathematical  Sciences,   Yeshiva
University \\
 New York, NY, 10033, USA }
 \address[mysecondaryaddress]{School of   Mathematics and Statistics,
Xinyang Normal University \\
Xinyang,    464000, China}

\begin{abstract}

 In this paper,  we first
 establish Hopf's lemmas for   parabolic fractional equations and parabolic fractional $p$-equations.  Then we derive an asymptotic Hopf's lemma for antisymmetric solutions  to parabolic fractional equations.
We believe that these  Hopf's lemmas will become powerful tools in    obtaining qualitative properties of   solutions for nonlocal parabolic equations.
\end{abstract}

\end{frontmatter}

\section{Introduction}
Hopf's   lemma is a classic result in analysis, dating back to the discovery of the maximum principles for harmonic functions \cite{H}, and it has become a fundamental and powerful tool in the study of partial differential equations \cite{B,LN,PW}.

In the past few decades,  elliptic equations involving either local or nonlocal operators have been extensively studied by many scholars, a number of systematic approaches have been established
to explore   qualitative properties of  solutions, such as \emph{the method of moving planes  in integral forms} \cite{CLO}, \emph{the extension method} \cite{CS}, \emph{the method of moving planes}
\cite{CL,CLL,CLLg,CLM,CLZ,CQ,CT,CZ1,DLL,DQ,L1,LZ2,LyL,WN,WW}, \emph{the method of moving spheres} \cite{CLZr,LZm}, and \emph{the sliding methods} \cite{LiuC,CW3,DQW,WC1}.

Hopf's lemma is a powerful tool in carrying out the method of moving planes to derive symmetry, monotonicity, and non-existence of solutions for elliptic partial differential equations.

Recently, with the extensive study of fractional  Laplacian and fractional $p$-Laplacian,  some  fractional version of Hopf's lemma  have been established. For instance, Li and Chen \cite{LC} introduced a fractional version of a Hopf's lemma for anti-symmetric functions which can be applied immediately to the method of moving planes to establish qualitative properties, such as symmetry and montonicity of solutions for fractional equations. Jin and Li \cite{JL} derived a Hopf's lemma for a fractional $p$-Laplacian.  Chen,  Li and Qi \cite{CLQ} obttained a Hopf's type lemma for positive weak super-solutions of the
fractional $p$-Laplacian equations with Dirichlet conditions.

So far as we aware, not much is known concerning Hopf type lemmas for parabolic equations involving nonlocal elliptic operators. This is a motivation of the present paper.
Here, we   investigate the following  parabolic fractional $p$-equations
$$ \frac{\partial u}{\partial t}+(-\lap)_p^su(x,t)=f(t,u(x,t)),~(x,t)\in \Omega \times (0,\infty),$$
where $0<s<1,~2\leq p<\infty,$ and  $\Omega$ is either a bounded or an unbounded domain in $\mathbb R^n$.
   For each fixed $t>0$,
 $$(-\lap)^s_p u(x,t)=C_{n,sp}P.V.\int_{\mathbb R^n} \frac{|u(x,t)-u(y,t)|^{p-2}(u(x,t)-u(y,t))}{|x-y|^{n+sp}}dy,$$
 where $P.V.$ stands for the Cauchy principal value. It is easy to see that    for $u \in C^{1,1}_{loc} (\mathbb R^n)\cap   {\cal L}_{sp } ,$ $(-\lap)_p^su$ is well defined, where $$ {\mathcal L}_{sp}=\{u(\cdot, t) \in L^{p-1}_{loc} (\mathbb{R}^n) \mid \int_{\mathbb R^N} \frac{|u(x,t)|^{p-1}}{1+|x|^{n+sp}}dx<+\infty\}.$$

In the special case when $p = 2$,  $(-\lap)^s_p$ becomes the well-known fractional Laplacian $(-\lap)^s $.
We first obtain  the following

\begin{thm}\label{thmo2}(A Hopf's lemma for parabolic fractional Laplacians and fractional $p$-Laplacians)
Assume that $u(x,t)\in (C^{1,1}_{loc}(\Omega)\cap {\cal L}_{sp})\times C^1((0,T])$ is a positive solution to
\begin{equation}\label{eq:o2}
\left\{\begin{array}{ll}
 \frac{\partial u}{\partial t}+(-\lap)_p^su(x,t)=f(t,u(x,t)),   &( x,t) \in \Omega\times (0,T],\\
 u(x,t)=0, & ( x,t) \in \Omega^c\times (0,T],
\end{array}
\right.\end{equation}
where $0<s<1,~p\geq2$ and $\Omega$ is a bounded domain in $\mathbb R^n$ with smooth boundary.
Assume that
\be \label{01}
  f(t,0)=0,~t\in (0,T),~  f ~\mbox{is Lipschitz continuous in}~   ~u ~\mbox{uniformly for}~t.
\ee
Suppose that at a   point $x_0 \in   \partial \Omega$, tangent to which  a sphere in $\Omega$ can be constructed.  Then there exists a positive constant $c_0$, such that  for any $t_0\in (0,T)$   and  for all $x$ near  the boundary of $\Omega$,    we have
$$ u(x,t_0)\geq c_0 d^s(x),$$
where $d(x)=dist(x,\partial \Omega) $. It follows that
$$\frac{\partial u}{\partial \nu^s}(x_0, t_0) <0 \; (\mbox{ may possibly be } -\infty), ~\forall~t_0\in (0,T),$$
where $\nu$ is the outward normal of $\partial \Omega$ at $x_0$ and $\frac{\partial u}{\partial \nu^s}$ is the derivative of fractional order $s$.
\end{thm}

One of the  applications of this kind of \emph{Hopf's lemma} is in the process of moving planes. To obtain a priori estimates of the solutions on a boundary layer (a neighborhood of $\partial \Omega$), or to prove symmetry of solutions if $\Omega$ is symmetric, one effective tool is the method of moving planes. To ensure that the moving plane has a starting point, we need to show that the solution is monotone decreasing
in the outward normal direction near the boundary. To this end, we can exploit an immediate conclusion from the above theorem:
$$ \frac{\partial u(x_0,t)}{\partial \nu} = -\infty, \;\; x_0 \in \partial \Omega, \; \forall \, t \in (0, T).$$
It is expected that, through some regularity estimates, one would be able to obtain that
$$ \frac{\partial u(x,t)}{\partial \nu} <0, \;\; \mbox{ for all } x \mbox{ near }  x_0 \in \partial \Omega, \; \forall \, t \in (0, T).$$
The readers can find more details at the end of Section 2.

Our second main result is
the {\em asymptotic Hopf's lemma for antisymmetric functions} which is an important in carrying out the method of moving planes. Before stating it, let's first introduce some relevant background notation.

Choose  any direction to be the $x_1$ direction.   Let
$$T_{\lambda} =\{x \in \mathbb{R}^{n}|\; x_1=\lambda, \mbox{ for some given } \lambda\in \mathbb{R}\}$$
be the moving planes,
$$\tilde \Sigma_{\lambda} =\{x \in \mathbb{R}^{n} | \, x_1>\lambda\}$$
be the region to the right of the plane,
$$ x^{\lambda} =(2\lambda-x_1, x_2, ..., x_n)$$
be the reflection of $x$ about the plane $T_{\lambda}$, and
$$u_\lambda(x,t)= u(x^\lambda,t) ~\mbox{and}~ w_\lambda (x,t)=u_\lambda(x,t)-u(x,t).$$

To study the asymptotic symmetry and monotonicity of solutions, we consider
 the wellknown $\omega$-limit set of $u$
$$ \omega(u):=\{\varphi  \mid \varphi=\lim u(\cdot,t_k) ~\mbox{for~some}~t_k\rightarrow \infty\},$$
with the limit in $C_0(\mathbb R^n)$. Under the very mild assumptions that both $u$ and $f$ are bounded, from the regularity result of \cite{XFR}, one can derive that $\omega(u)$ is  a nonempty compact subset of  $C_0(\mathbb R^n)$ and
$$ \underset{t\rightarrow \infty}{\lim} dist_{C_0(\mathbb R^n)} (u(\cdot,t), \omega(u))=0 .$$

For each $\varphi (x) \in \omega(u)$,
denote $$ \psi_\lambda(x)=\varphi(x^\lambda)-\varphi(x)=\varphi_\lambda(x)-\varphi(x).$$
Obviously, it is the $\omega$-limit of $w_\lambda (x,t)$.

We establish
\begin{thm}\label{lem4w} (Asymptotic Hopf's lemma for antisymmetric functions)
Assume that $ w_\lambda (x,t) \in(C^{1,1}_{loc}( \tilde \Sigma_\la)\cap {\mathcal L}_{2s} )\times C^1((0,\infty))$ is bounded and satisfies
\begin{equation}\label{1ww2020c}
\left\{\aligned
&\frac{\partial w_\lambda }{\partial t}+(-\Delta)^sw_\lambda(x,t)=c_\lambda(x,t)w_\lambda(x,t),  &  (x,t) \in  \tilde \Sigma_{\lambda} \times (0,\infty),\\
&w_\lambda(x^\lambda,t)=-w_\lambda(x,t),  & (x,t) \in  \tilde \Sigma_{\lambda} \times (0,\infty),\\
&\underset{t \rightarrow \infty}{\underline{\lim}} w_\lambda(x,t)\geq 0 , &  x  \in  \tilde \Sigma_{\lambda} ,
\endaligned \right.
\end{equation}
where $c_\lambda(x,t) $ is bounded uniformly  for $t$.
If $\psi_\la$ is nonnegative and  $\psi_\lambda>0$ somewhere in $\tilde \Sigma_\la$. Then $$ \frac{\partial \psi_\lambda}{\partial \nu}(x)<0,~x\in \partial \tilde \Sigma_\lambda,$$ where $\nu$ is an outward normal vector on $\partial \tilde \Sigma_\la$.
\end{thm}
This \emph{asymptotic Hopf's lemma}  can be applied   to the second step in the asymptotic method of moving planes instead of the \emph{asymptotic narrow region principle} to directly obtain that the plane can move a little bit toward the right as we will illustrate at the end of Section 3.

This  paper is organized as follows. In Section 2 we present  proofs of the Hopf's lemma for parabolic fractional Laplacians and fractional $p$-Laplacians based on maximum principles and proper construction of sub-solutions. Due to the full nonlinearity of fractional $p$-Laplacians, the proof is quite different than that of fractional Laplacians.

Section 3 is devoted to an asymptotic Hopf's lemma for antisymmetric solutions to parabolic fractional equations.

We belief that both Hopf lemmas will become useful tools in the analysis of qualitative properties of solutions to parabolic fractional and fractional $p$ equations.

\section{  Hopf's lemmas for parabolic fractional and fractional $p$-equations}

In this section, we prove Theorem \ref{thmo2}. Since the different natures of fractional Laplacian and the fractional $p$-Laplacian--the former is linear while the latter is fully non-linear, the proofs are different. Hence we present them in two separate subsections.

\subsection{A Hopf's lemma for parabolic fractional Laplacian}
In order to prove first part of Theorem \ref{thmo2} when $p=2$, we first obtain
\begin{thm}\label{thmoo1} (A maximum principle for parabolic fractional Laplacian) Let $\Omega$ be a bounded domain in $\mathbb R^n$.
Assume that $u (x,t)\in (C^{1,1}_{loc}(\Omega)\cap {\cal L}_{2s} )\times  C^{1} (   [0,\infty))$  and is lower semi-continuous   about $x$ on  $\bar \Omega$.    If
\begin{equation}\label{j1000}
\left\{\begin{array}{ll}
\frac{\partial u}{\partial t}  (x,t)+(-\lap)^{s}u (x,t) \geq 0,  & (x,t)\in \Omega\times ( 0,\infty),\\
u(x,t) \geq 0, & (x,t)\in (\mathbb R^n\backslash  \Omega) \times [0,T]   ,\\
u(x,0)\geq 0,&  x \in \Omega,
\end{array}
\right.
\end{equation}
then
\be \label{520}
u(x,t) \geq 0  ~\mbox{in}~  \bar \Omega \times [0,T],~\forall ~T>0 .
\ee
Furthermore,  \eqref{520} holds for unbounded region $\Omega$ if we further assume that for all $t\in [0,T]$
\be \label{521}
\underset{|x| \rightarrow \infty}{\underline{\lim }} u(x,t)\geq 0.
\ee
Under the conclusion   \eqref{520},  if  $u $   vanishes somewhere at   $ (x_0,t_0)\in \Omega \times (0,T]$, then
\begin{equation}\label{eq:j6}
u(x,t_0) = 0,~ \mbox{almost~ everywhere for}~ x \in \mathbb R^n.
\end{equation}
\end{thm}
\begin{proof}
 Fix an arbitrary $T\in (0,\infty)$.  If \eqref{520} does not hold, then the lower semi-continuity of $u$ on $\bar \Omega$  implies that there exists an $ \bar x \in   \Omega$ and $\bar t \in (0,T]$ such that
$$ u (\bar x,\bar t):= \underset{  \Omega \times (0,T]}{\min} u(x,t)<0.$$
It follows that $  \frac{\partial u}{\partial t}(\bar x,\bar t) \leq  0$.
Furthermore, by a direct calculation,
$$\aligned
(-\lap)^s u(\bar x,\bar t)&= C_{n,s} P.V. \int_{\mathbb R^n} \frac{u(\bar x,\bar t)-u(y,\bar t)}{|\bar x-y|^{n+2s}}dy \\
&=C_{n,s} P.V. \int_{\Omega} \frac{u(\bar x,\bar t)-u(y,\bar t)}{|\bar x-y|^{n+2s}}dy+C_{n,s}\int_{ \Omega^c} \frac{u(\bar x,\bar t)-u(y,\bar t)}{|\bar x-y|^{n+2s}}dy \\
&\leq C_{n,s}\int_{\Omega^c} \frac{u(\bar x,\bar t)-u(y,\bar t)}{|\bar x-y|^{n+2s}}dy \\
&<0,
\endaligned $$
the second inequality from the bottom holds because
$$u(y,\bar t) \geq 0,~y\in  \Omega ^c,~\bar t \in [0,T]$$
 by \eqref{j1000} and $u (\bar x,\bar t)< 0$.
Hence $$\frac{\partial u}{\partial t}(\bar x,\bar t) +(-\lap)^s u(\bar x, \bar t)<0, $$
which contradicts the first inequality in \eqref{j1000}.  Hence \eqref{520} must be valid.

If $\Omega$ is unbounded, then  \eqref{521} guarantees that the negative minima of $u(x,t)$ in $\Omega\times (0,T]$ must be attained at some points. Then one can follow the same discussion as the proof of \eqref{520}  to arrive at a contradiction.

 Next we prove \eqref{521} based on \eqref{520}.    Suppose that there exists $(x_0, t_0 )\in \Omega \times (0,T]$ such that
$$u(x_0, t_0 )=0.$$
It is obvious that $(x_0,t_0)$ is the minimum point of $u(x,t)$. Hence,
 $ \frac{\partial u}{\partial t}( x_0 ,t_0)\leq0$.
 Then from
$$0\geq \frac{\partial u}{\partial t}( x_0 ,t_0)=-(-\lap)^s u(x_0,t_0)=C_{n,s} \int_{\mathbb R^n}\frac{u(y,t_0)}{|x_0-y|^{n+2s}}dy$$
and $u(y,t) \geq 0$ in $\mathbb R^n \times [0,T]$, we obtain
$$u(x,t_0)=0~ \mbox{almost~everywhere~in}~\mathbb R^n.$$
\end{proof}

Next we will construct a subsolution to prove the Hopf's lemma for parabolic fractional Laplacian. For reader's convenience, we restate first part of  Theorem \ref{thmo2} as follows.
\begin{thm}  (A Hopf's lemma for parabolic fractional Laplacian)
Let $u (x,t)\in(C^{1,1}_{loc}(\Omega)\cap {\cal L}_{2s})\times C^1((0,T])$ be a positive solution to
\begin{equation}\label{1eq:o1}
\left\{\begin{array}{ll}
 \frac{\partial u}{\partial t}+(-\lap)^su(x,t)=f(t,u(x,t)), & (x,t)\in \Omega\times (0,T], \\
 u(x,t)=0, &  (x,t)\in \Omega^c\times (0,T],
\end{array}
\right.\end{equation}
where $0<s<1$ and $\Omega$ is a bounded domain in $\mathbb R^n$ with smooth boundary.
Assume that
\be \label{101}
  f(t,0)=0,~t\in (0,T),~  f ~\mbox{is Lipschitz continuous in}~   ~u ~\mbox{uniformly for}~t.
\ee
Suppose that at a   point $x_0 \in   \partial \Omega$, tangent to which  a sphere in $\Omega$ can be constructed.  Then there exists a positive constant $c_0$, such that  for any $t_0\in (0,T)$   and  for all $x$ near  the boundary of $\Omega$, we have
$$ u(x,t_0)\geq c_0 d^s(x),$$
where $d(x)=dist(x,\partial \Omega) $. It follows that
$$\frac{\partial u(x,t_0)}{\partial \nu^s}<0,~\forall~x\in \partial \Omega,~\forall~t_0\in (0,T),$$
where $\nu$ is the outward normal of $\partial \Omega$ at $x_0$.
\end{thm}
\begin{proof}  In the following Figure 1, consider the cylinder
\medskip

\begin{tikzpicture}
\draw (-2,4) -- (-2,0) arc (180:360:2cm and 0.5cm) -- (2,4) ++ (-2,0) circle (2cm and 0.5cm);
\draw[densely dashed] (-2,0) arc (180:0:2cm and 0.5cm);
\draw (0.4,3) -- (0.4,1) arc (180:360:0.8cm and 0.2cm) -- (2,3) ++ (-0.8,0) circle (0.8cm and 0.2cm);
\draw[densely dashed] (0.4,1) arc (180:0:0.8cm and 0.2cm);
\fill [orange]  (1.2,1) circle (0.8cm and 0.2cm);
\draw (1.2,2) circle (0.8cm and 0.2cm);
\draw (-1.5,3) -- (-1.5,1) arc (180:360:0.5cm and 0.1cm) -- (-0.5,3) ++ (-0.5,0) circle (0.5cm and 0.1cm);
\draw[densely dashed] (-1.5,1) arc (180:0:0.5cm and 0.1cm);
\path (2,0) [very thick,fill=black]  circle(1pt) node at (2.2,0) {$0$};
\path (2,1) [very thick,fill=black]  circle(1pt) node at (2.6,1) {$t_0-1$};
\path (2,2) [very thick,fill=black]  circle(1pt) node at (2.2,2) {$t_0$};
\path (2,3) [very thick,fill=black]  circle(1pt) node at (2.6,3) {$t_0+1$};
\path (2,4) [very thick,fill=black]  circle(1pt) node at (2.2,4) {$T$};
\node (test) at (3.5,2) {$P(x_0,t_0)$};
\draw [->,thin] (2,2) to [in = 90, out = 30] (3.5,2.2) (test);
\path (1,2.4) node at (1,2.4) {$B$};
 \path (4.8,4.3)  node [ font=\fontsize{10}{10}\selectfont] at (4.8,4.3)  {$B=B_1(0)\times [t_0-1,t_0+1]$};
\path (0,4.8) node at (0,4.8) {$E=\Omega\times (0,T]$};
\path (-1,1.5) node at (-1,1.5) {$\hat B$};
\node [below=1cm, align=flush center,text width=8cm] at (0,0.5)
        {$Figure$ 1 };
\end{tikzpicture}

Let $\partial B_\delta(\bar x)$ be a sphere in $\Omega$ that is tangent to $x_0$.  Use this, we construct a small circular cylinder
 $$B:=B_\delta(\bar x)\times [t_0-\epsilon,t_0+\epsilon] \subset \Omega \times (0,T),~\bar x \in \Omega,~t_0\in(0,T),$$
  where $\delta$ and $\epsilon$ are small positive constants. By a translation and a rescaling, for simplicity of notation, we may assume that   $B:=B_1(0)\times [t_0-1,t_0+1]$.

Let $ \hat B=D  \times [t_0-1,t_0+1]$ be a compact subset of $E$ which has a positive distance from $B$.  Since $u$ is positive and continuous, we have
 \be \label{cw2}
 u(x,t)\geq c_1>0,~(x,t)\in \hat B,
 \ee
for some constant $c_1$.

 To obtain a lower bound of $u$ in $B$, we construct a subsolution.  Set
 $$\underline{u}(x,t)=\chi _D (x)u(x,t)+\varepsilon (1-|x|^2)^s_+\eta(t),$$
  where $\varepsilon$ is a positive constant to be determined later,
 \be\label{cw9}
 \chi_D(x)=\left\{\begin{array}{ll}
1,& \quad  x \in  D,\\
 0,&\quad  x \not \in D^c ,
\end{array}
\right.\ee
 and  $\eta(t)\in C^{\infty}_0([0,T])$ satisfies
\begin{equation}\label{eq:o70}\eta(t)=\left\{\begin{array}{ll}
1,& \quad t\in [t_0-\frac{1}{2} , t_0+\frac{1}{2}],\\
 0,&\quad t\not\in (t_0-1,t_0+1).
\end{array}
\right.\end{equation}

By the definition of fractional Laplacian and \eqref{cw2}, we derive that for each fixed $t \in [t_0-1,t_0+1]$ and for any $x\in  B_1(0)$
\begin{equation}\label{eqo6}\aligned(-\lap)^s [\chi _D (x)u(x,t)]&=C_{n,s}P.V.\int_{\R^n}\frac{0-\chi _D (y)u(y,t)}{|x-y|^{n+2s}}dy\\
&=C\int_{D}\frac{-u(y,t)}{|x-y|^{n+2s}} dy\\
& \leq -C_1,~
\endaligned\end{equation}
where $C_1$ is a positive constant.

By \eqref{eqo6},    we obtain that for $(x,t)\in B_1(0)\times [t_0-1,t_0+1]$
 $$\aligned &\frac{\partial \underline{u}}{\partial t}+(-\lap)^s \underline{u}(x,t)\\
 &=\varepsilon \eta'(t)(1-|x|^2)^s_++(-\lap)^s [\chi _D (x)u(x,t)] +\varepsilon (-\lap)^s [(1-|x|^2)^s_+]\eta(t)\\
  &\leq -C_1+\varepsilon [\eta'(t)(1-|x|^2)^s+a \eta(t)],
  \endaligned $$
  where     the last inequality is due to \cite{G}
    $$ (-\lap)^s (1-|x|^2)^s_+ = a,~x\in B_1(0) ,~a
    ~\mbox{is a positive constant}.$$
 Denote
 $$w(x,t)=u(x,t)-\underline{u}(x,t).$$
Then $w(x,t) $ satisfies
 $$\aligned \frac{\partial w}{\partial t} + (-\lap)^s w(x,t)&\geq f(t,u)+C_1-\varepsilon (\eta'(t)(1-|x|^2)^s+a\eta(t))\\
 &   = C(x,t) u(x,t) +C_1-\varepsilon (\eta'(t)(1-|x|^2)^s+a\eta(t)),
  \endaligned$$
  where   $C(x,t)=\frac{f (t,u)-f(t,0)}{u(x,t)-0} =\frac{f (t,u)}{u(x,t)}$ is bounded by  \eqref{101}.  From $$u(x,t)=0,~(x,t)\in \Omega^c \times(0,T)$$
   and the continuity of  $u(x,t)$, one knows that $u(x,t)$ is small when $x\in \Omega$ is sufficiently close to $\partial\Omega$ and $t\in [t_0-\epsilon,t_0+\epsilon]$.
  Hence   taking $\varepsilon$ sufficiently small, we have
  $$
  \frac{\partial w}{\partial t}+(-\lap)^s w(x,t)
  \geq 0,~(x,t)\in B_1(0)\times [t_0-1,t_0+1].
 $$
 By the definition of $\underline{u}(x,t)$, we have
$$w(x,t)\geq 0, ~~(x,t) \in    B_1^c (0)\times [t_0-1,t_0+1]$$
and
$$
w(x,t_0-1)\geq 0,\ \  x \in   B_1 (0).
$$
Now, by applying
 Theorem \ref{thmoo1}, we derive
$$
w(x,t)\geq 0, ~(x,t)\in B_1(0)\times [t_0-1,t_0+1].
$$

Therefore
 $$ u(x,t) \geq   \underline{u}(x,t)=\varepsilon (1-|x|^2)^s_+ \eta (t),~\mbox{in}~B_1(0)\times [t_0-1,t_0+1].$$
In particular, fixed $\varepsilon$, by \eqref{eq:o70}, for $x \in B_1(0) $ one   has
\begin{equation}\label{eq:o10}
u(x,t_0)\geq \varepsilon (1-|x|^2)^s_+=\varepsilon ((1-|x|)^s(1+|x|))^s_+=c_0 d^s(x),
\end{equation}
where $d(x)=dist(x,\partial \Omega) $ and $c_0$ is a positive constant. Consequently,
$$\underset{x\rightarrow \partial \Omega}{\lim}\frac{  u(x,t_0)}{d^s(x)}\geq c_0>0,~\forall~t_0\in (0,T).$$

It follows that if $\nu$ is the outward normal of $\partial \Omega$ at $x_0$, then
 $$\frac{\partial u(x,t_0)}{\partial \nu^s}<0,~\forall~x\in \partial \Omega,~\forall~t_0\in (0,T),$$
since $u(x_0, t_0)=0$.
\end{proof}

\subsection{A Hopf's lemma for parabolic fractional $p$-Laplacian}
We first prove
\begin{thm}\label{thmoo2} (A maximum principle for parabolic fractional $p$-Laplacian) Let $p>2,~\Omega$ be a bounded domain in $\mathbb R^n$.
Assume that
 $$u (x,t), \underline{u}(x,t) \in (C^{1,1}_{loc}(\Omega)\cap {\cal L}_{sp} )\times  C^{1} (   [0,\infty))$$
  and $w(x,t)=u(x,t)-\underline{u}(x,t)$ is lower semi-continuous   about $x$ on  $\bar \Omega$.    If
\begin{equation}\label{1j1000}
\left\{\begin{array}{ll}
\frac{\partial w}{\partial t}  (x,t)+(-\lap)_p^{s}u (x,t) -(-\lap)_p^{s}\underline{u} (x,t) \geq 0,    ~ (x,t)\in \Omega\times ( 0,\infty),\\
w(x,t) \geq 0,  ~ \hspace{4.2cm}(x,t)\in (\mathbb R^n \backslash \Omega )\times [0,T]     ,\\
w(x,0)\geq 0,   ~ \hspace{6.8cm} x\in \Omega,
\end{array}
\right.
\end{equation}
then
\be \label{5201}
w(x,t) \geq 0  ~\mbox{in}~  \bar \Omega \times [0,T],~\forall ~T>0 .
\ee
  Furthermore,    \eqref{5201} hold for unbounded region $\Omega$ if we further assume that for all $t\in [0,T]$
\be \label{52110}
\underset{|x| \rightarrow \infty}{\underline{\lim }} u(x,t)\geq 0.
\ee
  Under the conclusion \eqref{5201}, if $w=0$ at some point $(x_0,t_0)\in \Omega \times (0,T]$, then
\be \label{5211}
w(x,t_0)=0,~\mbox{almost~everywhere~in}~\mathbb R^n.
\ee
\end{thm}
\begin{rem}
Here, the main difference between  parabolic fractional  $p$-Laplacian and parabolic fractional Laplacian is that the former is a non-linear operator, hence we need to use
$ (-\lap)_p^{s}u (x,t) -(-\lap)_p^{s}\underline{u} (x,t)$ instead of $(-\lap)_p^{s}w (x,t),$ and this makes the proof different.
\end{rem}
 \begin{proof} Fix an arbitrary $T\in (0,\infty)$.  If \eqref{5201} does not hold, then the lower semi-continuity of $w$ on $\bar \Omega$  implies that there exists an $ \bar x \in  \Omega$ and  a $\bar t \in (0,T]$ such that
$$ w (\bar x,\bar t):= \underset{  \Omega \times (0,T]}{\min} w(x,t)<0.$$
 Hence
$$  \frac{\partial w}{\partial t}(\bar x,\bar t) \leq 0.$$
Denote     $G(z)=|z|^{p-2}z$.  Obviously, $G(z)$ is strictly increasing, and $G'(z)=(p-1)|z|^{p-2}\geq 0$.
Through a direct calculation,
\begin{equation}\label{eq:o100}\aligned
&(-\lap)_p^s u(\bar x,\bar t)-(\lap)^s_p \underline{u}(\bar x, \bar t) \\
=& C_{n,sp} P.V. \int_{\mathbb R^n} \frac{ G(u(\bar x,\bar t)-u( y,\bar t))-G(\underline{u}(\bar x,\bar t)-\underline{u}( y,\bar t))}{|\bar x-y|^{n+sp}}dy \\
=&C_{n,sp} P.V. \int_{\Omega} \frac{ G(u(\bar x,\bar t)-u( y,\bar t))-G(\underline{u}(\bar x,\bar t)-\underline{u}( y,\bar t))}{|\bar x-y|^{n+sp}}dy \\
&+C_{n,sp}\int_{\mathbb R^n\backslash \Omega} \frac{ G(u(\bar x,\bar t)-u( y,\bar t))-G(\underline{u}(\bar x,\bar t)-\underline{u}( y,\bar t))}{|\bar x-y|^{n+sp}}dy \\
=&I_1+I_2.
\endaligned \end{equation}
In $I_1$, $y\in \Omega$, we derive
$$G(u(\bar x, \bar t)-u(y,\bar t))-G(\underline{u}(\bar x,\bar t)-\underline{u}(y, \bar t))\leq 0,$$
due to the monotonicity of $G$, $(\bar x, \bar t)$ is the minimum point of $w$ and the fact that
$$[u(\bar x, \bar t)-u(y,\bar t)]-[\underline{u}(\bar x,\bar t)-\underline{u}(y,\bar t)=w(\bar x,\bar t)-w(y, \bar t)\leq 0.$$
Hence
\begin{equation}\label{eq:o102}
I_1\leq 0.
\end{equation}
For $I_2$.
noticing that $w(y,\bar t)\geq 0,~y\in \Omega^c$ and $w(\bar x,\bar t)<0$, we obtain
$$u(\bar x,\bar t)-u(y,\bar t)-[\underline{u}(\bar x,\bar t)-\underline{u}(y,\bar t)]=w(\bar x,\bar t)-w(y,\bar t)<0,~y\in \Omega^c.$$
Then
\begin{equation}
I_2 < 0
\label{eq:o101}
\end{equation}
follows from the strict monotonicity of $G(\cdot)$.

Combining \eqref{eq:o100}, \eqref{eq:o102}, and  \eqref{eq:o101}, we arrive at
 $$(-\lap)^s_p u(\bar x,\bar t)-(-\lap)^s_p\underline{u}(\bar x,\bar t)<0.$$
Consequently $$\frac{\partial w}{\partial t}(\bar x,\bar t) +(-\lap)^s_p u(\bar x, \bar t)-(-\lap)^s_p \underline{u} (\bar x, \bar t)<0, $$
which contradicts the first inequality of \eqref{1j1000}.  Hence \eqref{5201}  holds.

If $\Omega$ is unbounded, \eqref{52110} guarantees that the negative minimum of $w(x,t)$  must be attained at some point, then one can follow the same discussion as in  the proof of \eqref{5201}  to arrive at a contradiction.

 Next we prove \eqref{5211} based on \eqref{5201}.    Suppose that there exists $(x_0, t_0 )\in \Omega \times (0,T]$ such that
$$w(x_0, t_0 )=0.$$
It is obvious that $(x_0,t_0)$ is the minimum point of $w(x,t)$. Hence,
 $ \frac{\partial w}{\partial t}( x_0 ,t_0)\leq0$.
 Then by \eqref{eq:o100}, \eqref{eq:o102} and \eqref{eq:o101},  one has
\be \label{555}\aligned
  \frac{\partial w}{\partial t}( x_0 ,t_0)+(-\lap)_p^s u(x_0,t_0)-(-\lap)_p^s \underline{u}(x_0,t_0) \leq C \int_{\mathbb{R}^n}\frac{-(w(y,t_0))^{p-1}}{|x_0-y|^{n+sp}}dy.
\endaligned \ee
Here the last inequality was derived from the following
\begin{lem}\label{lemo5}\cite{WC2}
For $G(z)=|z|^{p-2}z,~p>2$, there exists a constant $C>0$ such that
$$G(z_2)-G(z_1) \geq C(z_2-z_1)^{p-1},$$ for arbitrary $z_2\geq z_1$.
\end{lem}

Now combining  \eqref{1j1000},  \eqref{555} and
$w(y,t) \geq 0$ in $\mathbb R^n \times [0,T]$, we obtain
$$w(x,t_0)=0~ \mbox{almost~everywhere~in}~\mathbb R^n.$$
\end{proof}
For reader's convenience, we restate  Theorem \ref{thmo2} before its proof as follows.
\begin{thm} (A Hopf's lemma for parabolic fractional $p$-Laplacian)
Assume that $u(x,t)\in (C^{1,1}_{loc}(\Omega)\cap {\cal L}_{sp})\times C^1((0,T])$ is a positive solution to
\begin{equation}\label{1eq:o2}
\left\{\begin{array}{ll}
 \frac{\partial u}{\partial t}+(-\lap)_p^su(x,t)=f(t,u(x,t)),   &( x,t) \in \Omega\times (0,T],\\
 u(x,t)=0, & ( x,t) \in \Omega^c\times (0,T],
\end{array}
\right.\end{equation}
where $0<s<1,~p>2$ and $\Omega$ is a bounded domain in $\mathbb R^n$ with smooth boundary.
Assume that $f$ satisfies \eqref{01}.
 Suppose that at a   point $x_0 \in   \partial \Omega$, tangent to which  a sphere in $\Omega$ can be constructed.  Then there exists a positive constant $c_0$, such that  for any $t_0\in (0,T)$   and  for all $x$ near  the boundary of $\Omega$, we  have
$$ u(x,t_0)\geq c_0 d^s(x),$$
where $d(x)=dist(x,\partial \Omega) $. It follows that
 $$\frac{\partial u(x,t_0)}{\partial \nu^s}<0,~\forall~x\in \partial \Omega,~\forall~t_0\in (0,T),$$
 where $\nu$ is the outward normal of $\partial \Omega$ at $x_0$.
\end{thm}

\begin{proof}
Let $\partial B_\delta(\bar x)$ be a sphere in $\Omega$ that is tangent to $x_0$.  Use this, we construct a small circular cylinder
 $$B:=B_\delta(\bar x)\times [t_0-\epsilon,t_0+\epsilon] \subset \Omega \times (0,T),~\bar x \in \Omega,~t_0\in(0,T),$$
  where $\delta$ and $\epsilon$ are small positive constants.
  By translation and rescaling, for simplicity, we may assume that   $B:=B_1(0)\times [t_0-1,t_0+1]$. Also  see Figure 1.

Let $ \hat B=D  \times [t_0-1,t_0+1]$ be a compact subset of $\Omega \times (0,T]$ which has a positive distance from $B$.  Since $u$ is positive and continuous,  we  have \eqref{cw2}.

 Next, we construct a subsolution in $B_1(0)\times [t_0-1,t_0+1]$.  Set
 $$\underline{u}(x,t)=\chi _D(x) u(x,t)+\varepsilon \psi(x,t),$$
  where  $\varepsilon$ is a positive constant,  $$\psi(x,t)=(1-|x|^2)^s_+\eta(t),$$ $\chi_D(x)$ be as defined  in \eqref{cw9} and
$\eta(t)\in C^{\infty}_0( [0,T])$
  satisfies \eqref{eq:o70}.  By a result  in  \cite{LZ}, we have
  $$(-\lap)^s_p(1-|x|^2)^s_+\leq C, ~x\in B_1(0),$$
hence
  \begin{equation}\label{eqo5c}\aligned
  (-\lap)_p^s\psi(x,t)&= \eta^{p-1}(t)(-\lap)^s_p(1-|x|^2)^s_+\\
  &\leq C\eta^{p-1}(t),~(x,t)\in B_1(0)\times [t_0-1,t_0+1],
  \endaligned \end{equation}
where $C$ is a constant
and $0\leq \eta (t)\leq 1$.

 By the definition of $(-\lap)^s_p$,   \eqref{eqo5c} and lemma \ref{lemo5},    for each fixed $t \in [t_0-1,t_0+1]$ and for any $x\in B_1(0) $,   we obtain
  \begin{equation}\label{eq:o110}\aligned
  &(-\lap)^s_p \underline{u}(x,t)=(-\lap)^s_p[\chi_D(x)u(x,t)+\varepsilon \psi(x,t)]\\
  &=C_{n,sp}P.V.\int_{\mathbb R^n}\frac{G(\varepsilon \psi(x,t)-\chi_D(y)u(y,t)-\varepsilon\psi(y,t))}{|x-y|^{n+sp}}dy\\
  &=C_{n,sp}P.V.\{\int_{B_1(0)}\frac{G(\varepsilon \psi(x,t)-\varepsilon\psi(y,t))}{|x-y|^{n+sp}}dy+\int_{D} \frac{G(\varepsilon \psi(x,t)-u(y,t))}{|x-y|^{n+sp}}dy\\
 &~~ +\int_{\mathbb R^n\backslash (B_1(0)\cup D)}\frac{G(\varepsilon \psi(x,t))}{|x-y|^{n+sp}}dy
 +\int_{D}\frac{G(\varepsilon \psi(x,t))}{|x-y|^{n+sp}}dy -\int_{D}\frac{G(\varepsilon \psi(x,t))}{|x-y|^{n+sp}}dy\}\\
 &=(-\lap)^s_p (\varepsilon \psi(x,t))+\int_{D} \frac{G(\varepsilon \psi(x,t)-u(y,t))-G(\varepsilon \psi(x,t))}{|x-y|^{n+sp}}dy\\
 &\leq (-\lap)^s_p (\varepsilon \psi(x,t))-\int_{D} \frac{Cu^{p-1}(y,t) }{|x-y|^{n+sp}}dy\\
 &\leq C\varepsilon ^{p-1} \eta^{p-1}(t)   -C_1,
  \endaligned \end{equation}
where $C_1$ is a positive constant.

For $(x,t)\in B^c_1(0)\times [t_0-1,t_0+1]$, we have
$$
  u(x,t)\geq \chi_D(x) u(x,t) =\underline{u}(x,t).
 $$
It follows from \eqref{eq:o70} that
$$ u(x,t_0-1)\geq \underline{u}(x,t_0-1),~x\in B_1(0).$$

Denote
$$w(x,t)=u(x,t)-\underline{u}(x,t).$$
Combining \eqref{1eq:o2}  and \eqref{eq:o110}, for $(x,t)\in B_1(0)\times [t_0-1,t_0+1]$,  we have
$$\aligned
&\frac{\partial w}{\partial t} + (-\lap)_p^s u(x,t)-(-\lap)_p^s \underline{u}(x,t)\\
 &\geq f(t,u)-\varepsilon \eta'(t) (1-|x|^2)^s_+ -C \varepsilon ^{p-1} \eta^{p-1}(t)   +C_1\\
&= C(x,t) u -\varepsilon \eta'(t) (1-|x|^2)^s_+ -C \varepsilon ^{p-1} \eta^{p-1}(t)   +C_1,
\endaligned $$
where $C(x,t)=\frac{f(t,u)}{u(x,t)}$ is bounded by \eqref{01}.

 Similarly to the proof of the case when $p=2$ in Theorem  \ref{thmo2}. Taking $\varepsilon$ small, we obtain
$$\frac{\partial w}{\partial t} + (-\lap)_p^s u(x,t)-(-\lap)_p^s \underline{u}(x,t)
\geq 0,~(x,t)\in B_1(0)\times [t_0-1,t_0+1]. $$
Hence
$$
\left\{\begin{array}{ll}
\frac{\partial w}{\partial t} + (-\lap)_p^s u(x,t)-(-\lap)_p^s \underline{u}(x,t) \geq 0,~&(x,t)\in B_1(0)\times [t_0-1,t_0+1],\\
w(x,t) \geq 0,~ & (x,t)\in       B^c_1(0)\times [t_0-1,t_0+1] ,\\
w(x,t_0-1) \geq 0,~& x\in B_1(0)  .
\end{array}
\right.
$$
By Theorem \ref{thmoo2}, we derive
$$
w(x,t)\geq 0, ~\mbox{in }~B_1(0)\times [t_0-1,t_0+1].
$$
Therefore
\begin{equation}\label{eq:o10c}
u(x,t) \geq \varepsilon \psi(x,t)=\varepsilon (1-|x|^2)^s_+ \eta (t),~\mbox{in}~B_1(0)\times [t_0-1,t_0+1].
\end{equation}
In particular,
  $$u(x,t_0)\geq \varepsilon (1-|x|^2)^s_+=c_0d^s(x),~ x \in B_1(0) ,$$
where  $d(x)=dist(x,\partial\Omega ) $ and $c_0$ is a positive constant.

Hence if $\nu$ is the outward normal of $\partial \Omega$ at $x_0$, we obtain
 $$\frac{\partial u(x,t_0)}{\partial \nu^s}<0,~\forall~x\in \partial \Omega,~\forall~t_0\in (0,T),$$
since $u(x_0, t_0)=0$.

\end{proof}

Finally, we   briefly explain  how to apply Hopf's lemma for
 parabolic  Laplacian and parabolic $p$-Laplacian in the first step of the method of moving planes.
 Take $\Omega =B_1(0)$ as an example.
Let $$  \Omega_\la=\{x\in \Omega \mid x_1<\la\}.$$
To obtain the radial symmetry of positive solutions to
$$
\left\{\begin{array}{lll}
\frac{\partial u}{\partial t}+(-\lap)^s_p u(x,t)=f(t,u(x,t)), & (x,t)\in \Omega \times (0,T],\\
u(x,t)=0,& (x,t)\in \Omega^c \times (0,T],
\end{array}
\right.
$$
 where $0<s<1,~2\leq p < \infty$, the first step is to show that
for the plane $T_\lambda$ sufficiently close to the left end of  $\Omega$, $i.e.$, for $\lambda$ sufficiently close to  $$\inf \{  x_1 \mid  x \in \Omega\},$$ we have
\be \label{56}
u(x^\la,t)>u(x,t), ~ x\in \Omega_\la,~t\in  (0,T).
\ee
This will provide a starting point to move the plane.

By  applying the Hopf's lemma (Theorem \ref{thmo2}), we have
\be \label{567}
 \frac{\partial u}{\partial x^s_1}(\bar x,t)>0,\bar x\in \partial \Omega,~t\in (0,T).
 \ee
As a consequence, $$\frac{\partial u}{\partial x _1}(\bar x,t)=+\infty,~\bar x\in \partial \Omega,~t\in (0,T). $$
Then by the continuity of $\frac{\partial u}{\partial x^s_1}$ in some proper sense, it is natural to expect that
\be \label{990}
\frac{\partial u}{\partial x_1}(x,t)>0, \mbox{ for } ~x~ \mbox{sufficiently  close to }~\bar x ~ \mbox{in}~ \Omega,~t\in (0,T),
\ee
which implies \eqref{56} immediately.

We will prove \eqref{990} in a near future. While at this moment, the following two simple examples may shed some light on its validity.

It is wellknown that
$$(-\lap)^s (x_1)_+^s =0 \; \mbox{ in the right half space where } x_1>0,$$
and
$$ (-\lap)^s (1-|x|^2)^s_+ = \mbox{constant, in the unit ball } |x|<1.$$
Here $0<s<1$, and
$$(x_1)^s_+=
\left\{\begin{array}{lll} x_1^s, & x_1>0,\\
 0,&x_1\leq 0.
  \end{array}\right.
  $$

  Let $\eta(t)$ be any positive smooth function.

(a)  Consider
 $$\Psi_1(x,t)= (x_1)^s_+ \eta(t).$$
 It satisfies the equation
 $$\frac{\partial \Psi_1}{\partial t} + (-\lap)^s \Psi_1(x,t) = f_1(t, \Psi_1 (x,t)$$
 for some function $f_1$.

For $x=(x_1,x_2, \cdots,x_n)\in  \{x\in \mathbb R^n \mid x_1>0\},~t\in (0,T)$, we derive
$$ \frac{ \partial \Psi_1(x,t)}{\partial x_1}=\frac{ \partial (x_1)^s_+\eta(t)}{\partial x_1}
=\frac{s\eta(t)}{x_1^{1-s}}\rightarrow +\infty,~\mbox{as}~x_1 \rightarrow 0.$$
And obviously, for $x_1>0$, we have
$$\frac{ \partial \Psi_1(x,t)}{\partial x_1} >0. $$

(b) Then consider
 $$\Psi_2(x,t)=(1-|x|^2)^s_+\eta(t),$$
 then it satisfies the equation
  $$\frac{\partial \Psi_2}{\partial t} + (-\lap)^s \Psi_2(x,t) = f_2(t, \Psi_2 (x,t)$$
 for some function $f_2$.

For $x\in B_1(0)$ close to the left end of the region, i.e. $x_1$ close to $-1$, we have
$$ \frac{ \partial \Psi_2(x,t)}{\partial x_1}=\frac{ \partial (1-|x|^2)^s_+\eta(t)}{\partial x_1}
=\frac{-2 s x_1 \eta(t)}{(1-|x|^2)^{1-s}} \rightarrow +\infty,~\mbox{as}~|x| \rightarrow 1,$$
and at these points $x$,
$$\frac{ \partial \Psi_2(x,t)}{\partial x_1} > 0.$$

\section{Asymptotic Hopf's lemma for  antisymmetric solutions}
We start this section with the proof of the following
\begin{thm}  (Asymptotic Hopf's lemma for antisymmetric functions)
Assume that $ w_\lambda (x,t) \in(C^{1,1}_{loc}( \tilde \Sigma_\la)\cap {\mathcal L}_{2s} )\times C^1((0,\infty))$ is bounded and satisfies
\begin{equation}\label{ww2020c}
\left\{\aligned
&\frac{\partial w_\lambda }{\partial t}+(-\Delta)^sw_\lambda(x,t)=c_\lambda(x,t)w_\lambda(x,t),  &  (x,t) \in  \tilde \Sigma_{\lambda} \times (0,\infty),\\
&w_\lambda(x^\lambda,t)=-w_\lambda(x,t),  & (x,t) \in  \tilde \Sigma_{\lambda} \times (0,\infty),\\
&\underset{t \rightarrow \infty}{\underline{\lim}} w_\lambda(x,t)\geq 0 , &  x  \in  \tilde \Sigma_{\lambda} ,
\endaligned \right.
\end{equation}
where $c_\lambda(x,t) $ is bounded uniformly  for $t$.
If $\psi_\la$ is nonnegative and  $\psi_\lambda>0$ somewhere in $\tilde \Sigma_\la$. Then $$ \frac{\partial \psi_\lambda}{\partial \nu}(x)<0,~x\in \partial \tilde \Sigma_\lambda,$$ where $\nu$ is an outward normal vector on $\partial \tilde \Sigma_\la$.
\end{thm}
Recall that $\tilde \Sigma_\lambda$ is the region to the right of the plane $T_\lambda$.

\begin{proof}  Without  loss of generality, we may assume that $\lambda=0$, and it suffices to show that $\frac{\partial \psi_\lambda }{\partial x_1}(0)>0$.

For any  $\varphi(x)\in \omega(u)$, there  exists $t_k$ such that
$w_\lambda(x,t_k)\rightarrow \psi_\lambda(x)$ as $t_k\rightarrow \infty$.
Set $$w_k (x,t)=w_\lambda (x,t+t_k-1) .$$
Then $$
\frac{\partial w _k}{\partial t}+(-\lap)^s w _k (x,t)=c  _k(x,t)w _k(x,t),~(x,t)\in \tilde \Sigma_\la  \times [0,\infty),
$$
where $c_k(x,t)=c_\lambda(x,t+t_k-1)$ is bounded uniformly for $t$.
From regularity theory for parabolic equations \cite{XFR},  we conclude that there is a subsequence of $w_ k(x,t)$ (still denoted by $w_ k(x,t)$) which converges uniformly to a function $w_ \infty(x,t)$ in $\tilde \Sigma_\lambda\times[0,2]$
  and
  \begin{align*}
&\frac{\partial w_ k}{\partial t}(x,t)+(-\Delta)^sw_ k(x,t)\rightarrow \frac{\partial w_ \infty}{\partial t}(x,t)+(-\Delta)^sw_ \infty(x,t),\\
&c_k(x,t)\rightarrow c_\infty(x,t),\ \ \ \ \mbox{as}\ k\rightarrow\infty.
\end{align*}
Hence
\begin{equation}\label{eq:wc31}
 \frac{\partial w _\infty}{\partial t}+(-\lap)^s w _\infty (x,t)=c  _\infty(x,t)w _\infty(x,t),~(x,t)\in \tilde \Sigma_\la  \times [0,2].
\end{equation}
In particular,
$$w_\lambda(x,t_k)=w_k(x,1)\rightarrow w_\infty(x,1)=\psi_\lambda(x)\, \ \mbox{as}\ k\rightarrow \infty.$$

Let $$\tilde w(x,t)=e^{mt} w_\infty(x,t),~m>0.$$
Since $c_\infty$ is bounded, we can  choose $m$ such that
   \be \label{633}
   m   +c _\infty(x,t) \geq 0.
   \ee
By the third condition in \eqref{ww2020c}, we have
   \be \label{635}
   w_\infty (x,t) \geq 0, ~(x,t) \in \tilde \Sigma_\la \times [0,2].
\ee
Combining   \eqref{eq:wc31},   \eqref{633} and \eqref{635}, one has
\begin{equation}\label{eq:wcc66}
  \frac{\partial \tilde w}{\partial t}+(-\lap)^s \tilde w (x,t)
=   (m+c _\infty(x,t) ) \tilde w (x,t) \geq 0,~(x,t) \in\tilde  \Sigma_\la  \times [0,2].
 \end{equation}

 Since $\psi_\lambda>0$ somewhere in $\tilde \Sigma_\lambda$, by continuity, there exists a set $D\subset \subset \tilde\Sigma_{\lambda}$ such that
\begin{equation}\label{meq01}
\psi_\lambda(x)>c>0,\ \, \  x \in D,
\end{equation}
with positive constant $c$. We may assume  $ dist(T_0,\partial D)>2\varepsilon$,  where $\varepsilon$ is a positive constant.  Let $B_\varepsilon(0)$ and $B_{2\varepsilon}(0)$ be balls with the origin as the center and radius of $\varepsilon$ and $2\varepsilon$, respectively.

By the continuity of $w_\infty(x,t)$ (see \cite{XFR}), there exist $0<\varepsilon_o<1$, such that
$$ w_\infty (x,t) > c/2, \;\; (x,t) \in D \times [1-\varepsilon_o,1+\varepsilon_o].$$
For simplicity of notation, we may assume that
\be w_\infty (x,t) > c/2, \;\; (x,t) \in D \times [0,2].
\label{A50}
\ee
  Let    $D_\lambda$  be the reflection of   $D$ about the plane $T_\lambda$ for any time $t$.
  For convenience, for each $t\in [0,2]$,  we show the following Figure 2.
\begin{center}
\begin{tikzpicture}[node distance = 0.3cm]
\draw [->, semithick] (-4,0) -- (4,0) node[right] {$x_1$};
\draw [ semithick] (0,-2.4) -- (0,3) node[above] {$T_0$};
\path (2.3,2.7) node [ font=\fontsize{10}{10}\selectfont] {$\tilde\Sigma_0=\{ x\in \mathbb R^n\mid x_1>0\}$};
\draw (0,0) circle (0.8);
\draw (0,0) circle (1.8);
\fill[blue!25] (0,-1.8) arc(-90:90:1.8) node [black] [below right] at (0.2,-1.7 ) {$ B_{2\varepsilon}(0)\cap \tilde\Sigma_0 $};
\node (test) at (-2,2.7) {$B_{2\varepsilon}(0)$};
    \draw [->,thin] (-1.4,1) to [in = 0, out = 30] (-2,2.4) (test);
\node (test) at (-1.9,-1.8) {$B_\varepsilon(0)$};
    \draw [->,thin] (-0.6,-0.4) to [in = 90, out = 90] (-1.8,-1.6) (test);
\draw (-3,1) ellipse  [x radius=1cm, y radius=0.5cm];
\fill [orange] (-3,1) ellipse  [x radius=1cm, y radius=0.5cm];
\draw (3,1) ellipse  [x radius=1cm, y radius=0.5cm];
\fill [orange] (3,1) ellipse  [x radius=1cm, y radius=0.5cm];
\path (-2.5,1) node [ font=\fontsize{10}{10}\selectfont] {$D_\lambda$};
\path (2.5,1) node [ font=\fontsize{10}{10}\selectfont] {$D$};
\node [below=1cm, align=flush center,text width=8cm] at (0,-1.3)
        {$Figure$ 2 };
\end{tikzpicture}
\end{center}
 Denote $g(x)=x_1 \zeta(x) $, where
 $$
  \zeta(x)= \zeta(|x|)= \left\{
\begin{array}{ll} 1,~ & |x|<\varepsilon, \\
0,~& |x|\geq 2\varepsilon,\end{array}
\right.$$
 and
 $$0\leq \zeta(x) \leq 1,~  \zeta(x) \in C^\infty_0 (B_{2\varepsilon}(0)).$$
 Obviously, $g(x)$ is an anti-symmetric function with respect to plane $T_0$, i.e.
 $$g(-x_1, x_2, \cdots x_n) = - g(x_1, x_2, \cdots x_n).$$

Denote
$$\underline{w}(x,t)=\chi_{D\cup D_\lambda}(x) \tilde w(x,t)+\delta \eta(t)g(x),$$
where
 $$\chi_{D\cup D_\lambda}(x)=\left\{\begin{array}{ll}
1,& \quad x\in D\cup D_\lambda,\\
 0,&\quad x\not\in D \cup D_\lambda,
\end{array}
\right.$$
 and
 $\eta(t)\in C^\infty_0([1-\varepsilon_o,1+\varepsilon_o])$ satisfies
\be \label{eq:wco70}
  \eta(t)=\left\{\begin{array}{lll} 1,& t \in [1-\frac{\varepsilon_o}{2}, 1+ \frac{\varepsilon_o}{2}],\\
 0,&  t\not \in [1-\varepsilon_o,1+\varepsilon_o].
  \end{array}\right.\ee
Since $g(x)$ is a $C^\infty_0 (B_{2\varepsilon}(0))$ function, we have
\begin{equation}\label{eq:w4}
|(-\lap)^sg(x)|\leq C_0.
\end{equation}
By the definition of fractional Laplacian  and \eqref{A50}, we derive that for each fixed $t \in [0,2]$ and for any $ x  \in  B_{2\varepsilon}(0) \cap \tilde \Sigma _\la$
\begin{align}\label{eq:cw111}
&(-\lap)^s(\chi_{D\cup D_\lambda}\tilde{w}(x,t))\nonumber\\
&= C_{n,s} P.V. \int_{\mathbb R^n}\frac{\chi_{D\cup D_\lambda}(x)\tilde{w}(x,t)-\chi_{D\cup D_\lambda}(y)\tilde{w}(y,t)}{|x-y|^{n+2s}}\;dy\nonumber\\
&= C_{n,s} P.V. \int_{\mathbb R^n}\frac{-\chi_{D\cup D_\lambda}(y)\tilde{w}(y,t)}{|x-y|^{n+2s}}\;dy\nonumber\\
&= C_{n,s} P.V. \int_{D}\frac{-\tilde{w}(y,t)}{|x-y|^{n+2s}}\;dy+\int_{D}\frac{-\tilde{w}(y^\lambda,t)}{|x-y^\lambda|^{n+2s}}\;dy\nonumber\\
&= C_{n,s} P.V. \int_{D}\left(\frac{1}{|x-y^\lambda|^{n+2s}}-\frac{1}{|x-y|^{n+2s}}\right)\tilde{w}(y,t)\;dy\nonumber\\
&\leq -C_1,
\end{align} where $C_1$ is a positive constant.
 For $(x,t)\in (B_{2\varepsilon}(0) \cap \tilde \Sigma_\la)\times [0,2]$, by \eqref{eq:w4} and \eqref{eq:cw111}, we obtain
$$\aligned   \frac{\partial \underline{w}}{\partial t} +(-\lap)^s \underline{w}(x,t)
=&\delta \eta'(t) g(x)+(-\lap)^s( \chi_{D\cup D_\lambda}\tilde{w}(x,t))+ \delta \eta(t)(-\lap)^s g(x)
 \\
\leq& \delta \eta'(t) g(x) -C_1+ \delta \eta(t) C_0 .
\endaligned $$
Hence, taking $\delta$ sufficiently small, we derive
\begin{equation}\label{eq:wc33}\aligned
  \frac{\partial \underline{w}}{\partial t} +(-\lap)^s \underline{w}(x,t) \leq0 ,~ (x,t) \in (B_{2\varepsilon}(0) \cap \tilde\Sigma_\la )\times [0,2].
\endaligned \end{equation}
Set $$v(x,t)=\tilde{w}(x,t)-\underline{w}(x,t).$$
Obviously, $v(x,t)=-v(x^\lambda,t)$. From  \eqref{eq:wcc66} and \eqref{eq:wc33}, we derive  that
$v(x,t)$ satisfies
\begin{equation}\label{cw1}
\frac{\partial v}{\partial t}(x,t)+(-\Delta)^sv(x,t) \geq0,~(x,t) \in (B_{2\varepsilon}(0) \cap\tilde \Sigma_\la )\times[0,2].
\end{equation}
Also, by the definition of $\underline{w}(x,t)$, we have
$$v(x,t)\geq 0,\ \ (x,t) \in (\tilde\Sigma_\lambda\setminus (B_{2\varepsilon}(0) \cap \tilde\Sigma_\la))\times[0,2]$$
and $$
v(x,0)\geq 0,\ \  x\in \tilde \Sigma_\lambda.
$$
Now,  we apply the following lemma to $v(x,t)$.
\begin{lem}(Maximum principle for antisymmetric functions)\cite{CWNH}
Let $\Omega$ be a bounded domain in $\Sigma_\la$. Assume that   $w_\la(x,t)\in (C^{1,1}_{loc}(\Omega)\cap {\mathcal L}_{2s})\times C^1([0,\infty))$ is lower semi-continuous in $x$ on $\bar{\Omega}$ and satisfies
 \be \label{cw2029}
 \left\{
\begin{array}{ll}
 \frac{\partial w_\la}{\partial t}(x,t)+(-\lap)^s w_\la(x,t)\geq
 c_\la(x,t)w_\la(x,t), & (x,t)\in \Omega \times (0,\infty),\\
 w_\la(x^\la,t)=-w_\la(x,t), &  (x,t)\in  \Sigma_\la  \times [0,\infty),\\
     w_\la(x,t)\geq 0, & (x,t)\in ( \Sigma_\la \backslash \Omega )\times [0,\infty),\\
 w_\la(x,0) \geq 0, &  x \in  \Omega .
 \end{array}
\right.\ee
  If $c_\la(x,t) $ is bounded from above, then
\be \label{eq:cw2033}
w_\la(x,t)\geq 0,~(x,t)\in \Omega\times[0,T],\ \ \forall\ T>0.
\ee
\end{lem}

As an immediate cossequence of this lemma with $c_\la(x,t)=0$, we obtain
 $$
v(x,t)\geq0,\ \ (x,t) \in  (B_{2\varepsilon}(0) \cap \tilde \Sigma_\la)\times[0,2].
$$
It implies that
$$e^{mt}w_\infty(x,t)-\delta g(x)\eta(t)\geq0, \ \,\ (x,t) \in  (B_{2\varepsilon}(0) \cap \tilde\Sigma_\la)\times[0,2].$$
In particular,
$$ w_\infty(x,1)\geq e^{-m}\delta g(x) ,\ \, \  x \in  B_{2\varepsilon}(0) \cap \tilde \Sigma_\la$$
 and
 $$ w_\infty(x,1)\geq e^{-m}\delta x_1  ,\ \, \  x \in  B_{\varepsilon}(0) \cap \tilde \Sigma_\la.$$
Since
$w_\infty(x,1)\equiv 0,~x\in T_0$, particular, $w_\infty(0,1)=0$.
Hence
$$\frac{w_\infty(x,1)-0}{x_1-0}\geq e^{-m} \delta   >0 ,~ x \in B_\varepsilon(0) \cap \tilde \Sigma_\la.$$
Obviously, no matter how small $\varepsilon$ is, we have
$$
 \frac{\partial \psi_\lambda }{\partial x_1}(0)>0.
$$
Therefore,
$$ \frac{\partial \psi_\lambda }{\partial \nu}(x)<0,~x\in \partial \tilde \Sigma_\la.$$
\end{proof}

Finally, to illustrate how the asymptotic Hopf's lemma for anti-symmetric functions can be employed in the second step of the asymptotic method
of moving planes, we consider the following  example:
\begin{equation}\label{eqcw1}
\left\{\begin{array}{lll}
\frac{\partial u}{\partial t}+(-\lap)^s u(x,t)=f(t,u(x,t)), & \mbox{} (x,t) \in \mathbb R^n \times (0,\infty),\\
u(x,t)>0,& \mbox{} (x,t) \in \mathbb R^n\times (0,\infty) .
\end{array}\right.
\end{equation}

Under certain conditions on $f $  and  $\underset{|x|\rightarrow \infty }{\lim}~ u(x,t) =0 $, we want to show that   positive bounded solutions are
asymptotically symmetric about some point in $\mathbb R^n$. That is, all  $\varphi(x) \in \omega(u)$  are radially symmetric and decreasing about some point in $\mathbb R^n$.

To compare the values of $u(x,t)$ with $u(x^{\lambda},t)$, let
$$ w_\lambda (x,t) = u(x^\lambda, t) - u(x,t). $$
Obviously, $w_\lambda(x,t)$     satisfies
$$\frac{\partial w_\lambda}{\partial t}(x,t)+(-\Delta)^s w_{\lambda}(x,t) = c_{\lambda}(x,t) w_{\lambda}(x,t) , \;\; x \in \Sigma_{\lambda}.$$
For each $\varphi (x) \in \omega(u)$,
denote $$ \psi_\lambda(x)=\varphi(x^\lambda)-\varphi(x)=\varphi_\lambda(x)-\varphi(x),$$
which is an $\omega$-limit of $w_\lambda (x,t)$.

To obtain the asymptotic symmetry of solutions to \eqref{eqcw1} in the  whole space, the first step is to show that
for $\lambda$ sufficiently close to either $-\infty$ or $\infty$,   we have
\be \label{eq:jj5}
\psi_\lambda(x)\geq 0,~x\in \Sigma_\lambda.
\ee
This provides a starting position to move the plane.

In the second step, we move the plane $T_\lambda$ to the right as long as inequality  \eqref{eq:jj5}  holds to its rightmost limiting position $T_{\lambda^-_0}$, where
$$\lambda_0^-=\sup\{\lambda \mid   \psi_\mu(x) \geq  0,~\forall~x\in \Sigma_\mu,~\mu \leq \lambda\}. $$
To show that there is at least one $\varphi \in \omega(u)$ which is symmetric about the plane $T_{\lambda^-_0}$,  or
\begin{equation}\label{w0}
 \psi_{\lambda^-_0}(x) \equiv 0 , \;\; x \in \Sigma_{\lambda^-_0},
 \end{equation}
one usually uses a contradiction argument. Suppose (\ref{w0}) is false, then for all $\varphi \in \omega(u)$,
$$ \psi_{\lambda^-_0}(x) > 0 , \;\; \mbox{ somwhere in } \; \Sigma_{\lambda^-_0}.$$
It follows from the {\em asymptotic Hopf lemma for anti-symmetric functions}, we have
\begin{equation}
\frac{\partial \psi_{\lambda^-_0}}{\partial \nu} (x^0) < 0
\label{A10}
\end{equation}
for any point $x^0$ on the boundary of $\Sigma_{\lambda^-_0}$.

Furthermore,  applying an  \emph{asymptotic strong maximum principle for antisymmetric function} (Theorem 6 in \cite{CWNH}), we have
\be\label{229}
\psi_{\lambda^-_0} (x) > 0 , \;\; \forall \, x \in \Sigma_{\lambda^-_0},~\forall~\varphi \in \omega(u) .
\ee

On the other hand,  by the definition of $\lambda^-_0$, there exists a sequence $\lambda_k \searrow\lambda^-_0$, $x^k \in \Sigma_{\lambda_k}$, and $\psi^k_{\la_k} $ (corresponding to $\varphi^k \in \omega(u)$), such that
\begin{equation}
\psi^k_{\lambda_k}(x^k) = \min_{\Sigma_{\lambda_k}} \psi_{\lambda_k}(x)  < 0  \; \mbox{ and } \; \grad \psi^k_{\lambda_k}(x^k) = 0 .
\label{wxk}
\end{equation}

Under the assumption of  $f_u(t,\cdot)<-\sigma(>0)$, by \emph{asymptotic maximum principle near infinity} (Theorem 4 in \cite{CWNH}), one knows that $\{x^k\}$ is bounded. Then it implies  that   $\{x^k\}$   converges to some point $x^0 $.
Due to the compactness of $\omega(u)$ in $C_0(\mathbb R^n)$, there exists $\psi^0_{\lambda^-_0} $ (corresponding to some $\varphi^0\in \omega(u)$),  such that
$$ \psi^k_{\la_k}(x^k)\rightarrow \psi^0_{\la^-_0}(x^0)~\mbox{as}~k\rightarrow \infty.$$
 Hence from (\ref{wxk}), we have
\be \label{006}
\psi^0_{\lambda^-_0}(x^0) \leq 0
\ee
and
\be \label{007}
\nabla \psi^0_{\lambda^-_0}(x^0) = 0 .
\ee
 It follows from  \eqref{229} and \eqref{006} that $x^0 \in \partial \Sigma_{\lambda^-_0}$.
Now \eqref{007} contradicts (\ref{A10}). Therefore (\ref{w0}) must be valid.

\end{document}